\documentclass[11pt]{article}
\usepackage{amsmath,amssymb,amsthm,amscd,dsfont}
\usepackage{amsfonts,latexsym,rawfonts,amsmath,amssymb,amsthm}
\usepackage{lscape}
\usepackage{amscd, float,times,rotating}
\usepackage{pb-diagram}
\usepackage{hyperref}
\usepackage{scalerel} 
\usepackage{graphicx} 
\numberwithin{equation}{section}

\newtheorem{prop}{Proposition}[section]
\newtheorem{theorem}[prop]{Theorem}
\newtheorem{lemma}[prop]{Lemma}
\newtheorem{corollary}[prop]{Corollary}

\newtheorem{example}[prop]{Example}

\begin{document}
\title{Lower bound for the first eigenvalue of $p-$Laplacian and applications in asymptotically hyperbolic Einstein manifolds}
\author{Xiaoshang Jin\thanks{The author was supported by {NNSF of China \# 12201225}}}

\date{}
\maketitle
\begin{abstract}
This paper investigates the first Dirichlet eigenvalue for the $p$-Laplacian in Riemannian manifolds. Firstly, we establish a lower bound for this eigenvalue under the condition that the domain includes a specific function which fulfills certain criteria related to divergence and gradient conditions. In the subsequent section, we introduce an enhanced lower bound for the eigenvalue, which is linked to the distance function defined in the domain. As a practical application, we provide an estimation for the first Dirichlet eigenvalue of geodesic balls with large radius in asymptotically hyperbolic Einstein manifolds.
\end{abstract}
\section{Introduction}
Suppose that $(M,g)$ is a complete Riemannian manifold. For any $p>1,$ we define the $p-$Laplacian
as
\begin{equation}\label{1.1}
\Delta_p : W^{1,p}_{\text{loc}}(M)\rightarrow W^{-1,q}(M),\ \ \ \ \ \Delta_p u=\text{div}(|\nabla u|^{p-2}\nabla u)
\end{equation}
Here $W^{-1,q}(M)$ represents the dual space of $W^{1,p}_0(M)$ and
the Sobolev space $W^{1,p}_0(M)$  is the closure $C^\infty_0(M)$ with respect to the norm
$$\|u\|_{1,p}=\left[\int_M(|u|^p+|\nabla u|^p)dv_g\right]^{\frac{1}{p}}.$$
    If $p=2,$ then
$\Delta_2=\Delta$
is the Laplace-Beltrami operator. We say that
$$\Delta_p u=0 \ (\geq 0,\leq 0)$$
if for all nonnegative function $\varphi\in C^\infty_0(M),$
$$-\int_M|\nabla u|^{p-2}\cdot g(\nabla\varphi,\nabla u) d\upsilon_g=0 \ (\geq 0,\leq 0).$$
\par Let $\Omega\subseteq M$ be a bounded domain with piecewise smooth boundary. The Dirichlet eigenfunctions are defined by solving the following problem for $u\neq 0$ and eigenvalue $\lambda$ as follows:
 \begin{equation}\label{1.2}
 \left\{
    \begin{array}{l}
    \Delta_p u=-\lambda |u|^{p-2} u \ \ \ \ \text{in} \ \Omega,
    \\ u=0\ \ \ \ \ \ \ \ \ \ \ \ \ \ \ \ \ \ \ \ \ \ \ \ \ \text{on}\ \partial \Omega.
    \end{array}
 \right.
 \end{equation}
 The first eigenvalue
$\lambda_{1,p}(\Omega)$ of the $p-$Laplacian is defined as the least number $\lambda$ for which there exists a nonzero function $u\in W^{1,p}_0(\Omega)$ that solves the equation (\ref{1.2}). It is well-known that $\lambda_{1,p}$ is associated
to an eigenfunction which is positive in $C^{1,\alpha}(\overline{\Omega})$ and is unique up to a
multiplicative constant (see  \cite{barles1988remarks} or \cite{belloni2002direct} for a simple proof in Euclidean space).
 It can be also characterized by  the relation \cite{henrot2006extremum}:
\begin{equation}\label{1.3}
\lambda_{1,p}(\Omega)=\inf\left\{\frac{\int_\Omega|\nabla u|^p d\upsilon_g}{\int_\Omega |u|^pd\upsilon_g}:u\in W^{1,p}_0(\Omega)\setminus\{0\}\right\}
\end{equation}
For a complete noncompact manifold $(M,g)$ the $p-$eigenvalue of $M$ can be defined as the limit:
$$\lambda_{1,p}(M)=\lim\limits_{k\rightarrow\infty}\lambda_{1,p}(\Omega_k)$$
for any smoothly compact exhaustion $\{\Omega_k\}_{k=1}^\infty$ of $M.$ The definition is well-defined because the first Dirichlet eigenvalue of $p-$Laplacian also has the property of domain monotonicity. One can see lemma 1.1 in \cite{du2017estimates} for more details.
\\
\par Our first result of this paper provides a lower bound of the first eigenvalue of
the $p-$Laplacian when the domain $\Omega$ admits a special function.  More specifically, we state the following theorem:
\begin{theorem}\label{thm1.1}
  Let $\Omega$ be a bounded smooth domain in a Riemannian manifold. Assume $p_2>p_1-1>0$ and $p=\frac{p_2}{p_2-p_1+1}.$ If $f\in W^{1,p_1}_{\rm loc}(\Omega)$ satisfies
  \begin{equation}\label{1.4}
    \Delta_{p_1} f-C|\nabla f|^{p_2}\geq D
  \end{equation}
  in $\Omega$ for some positive numbers $C$ and $D.$ Then
  \begin{equation}\label{1.5}
  \lambda_{1,p}(\Omega)\geq \Big(\frac{C}{p-1}\Big)^{p-1}\cdot D
  \end{equation}
 \end{theorem}
The estimate (\ref{1.5}) is scaling invariant in the following sense. For any $k>0,$ setting $\bar{f}=kf$ and $\bar{C}=Ck^{p_1-p_2-1}$ leads to
 $$\Delta_{p_1} \bar{f}-\bar{C}|\nabla \bar{f}|^{p_2}=k^{p_1-1}\Delta_{p_1} f-Ck^{p_1-p_2-1}k^{p_2} |\nabla f|^{p_2}\geq k^{p_1-1}D:=\bar{D},$$
 then $\lambda_{1,p}(\Omega)\geq\Big(\frac{\bar{C}}{p-1}\Big)^{p-1}\cdot\bar{D}= \Big(\frac{C}{p-1}\Big)^{p-1}\cdot D.$
\\
\par It should be noted that although the proof of Theorem \ref{thm1.1} is simple by Young's inequality with a traditional method, we believe that the theorem is still widely useful. In fact, appropriate constants $p_1$, $p_2$, $C$, $D$ can be chosen to develop other methods for estimating the lower bound of $\lambda_{1,p}(\Omega)$. Including:
\begin{itemize}
  \item [(1)] a new proof of Barta's inequality (see Proposition \ref{prop2.1});
  \item [(2)] an improved estimate of Theorem 1.1 in \cite{carvalho2022fundamental} (see Theorem \ref{thm2.2});
  \item [(3)] a new proof of ``$p-$eigenvalue for AHE manifolds" (see Proposition \ref{prop2.5}).
\end{itemize}
 We will provide more details in Section \ref{sec2.1}.
\\~
\par In the second part of the paper,we explore the asymptotical behavior of the first eigenvalue
$\lambda_{1,p}(\Omega)$ of the $p-$Laplacian as $\Omega$ expands to encompass the noncompact manifold. Initial findings by Savo in \cite{savo2009lowest} showed that
$$
  \lambda_{1,2}(B(o,R))=\frac{n^2}{4}+\frac{\pi^2}{R^2}+O(R^{-3}),\ \ R\rightarrow+\infty.
$$
for any geodesic ball $B(o,R)$ in an $n+1-$dimensional hyperbolic space. Later the result was extended in \cite{kristaly2022new} where the first four terms in the expansion of $\lambda_{1,2}(B(o,R))$ was obtained.  A recent study in \cite{jin2024asymptotic} has confirmed that these results are applicable even to the $n+1-$dimensionall asymptotically hyperbolic Einstein manifold with nonnegative Yamabe conformal infinity.
\par In the followings, we try to derive a similar estimate for the $p-$Laplacian, i.e. how $\lambda_{1,p}(B(o,R))$ tends to
$\big(\frac{n}{p}\big)^p$ as $R\rightarrow\infty$ in AHE manifolds with nonnegative Yamabe conformal infinity. We first present the following theorem, which can be considered as an enhancement of Theorem \ref{thm2.2}:
\begin{theorem}\label{thm1.2}
Let $\Omega$ be a bounded smooth domain in a complete Riemnnian manifold, if there exists a function $r:\Omega\rightarrow[0,R]$ satisfying that $|dr|=1$ and $\Delta r\geq k$ almost everywhere for some positive constants $R$ and $k,$ then for any $p>1,$
   \begin{equation}\label{1.6}
     \lambda_{1,p}(\Omega)\geq
       \Big(\frac{k}{p}\Big)^p\Big(1+\frac{\pi^2}{(1+\frac{k}{p}R)^2}\Big)^{\min\{p-1,1\}}.
   \end{equation}
\end{theorem}
Notice that the condition ``$|dr|=1$ and $\Delta r\geq k$ " would imply that $\lambda_{1,p}(\Omega)\geq \big(\frac{k}{p}\big)^p$ by Theorem \ref{thm2.2}. We now introduce a refined lower bound for of $\lambda_{1,p}(\Omega)$ that depends on $R=\sup\limits_{x,y\in\Omega} [r(x)-r(y)],$ the ``radius" of $\Omega$ in some sense.
Specifically,
 $$\lambda_{1,p}(\Omega)\geq \Big(\frac{k}{p}\Big)^p+\frac{C(p,k)}{R^2}+O(R^{-3})$$
 as $R$ tends to infinity.
 \par If we set $r$ to be the distance function of a point or a zero measure set, we can derive the lower bound for $\lambda_{1,p}$ for some special manifolds. For example, by applying the Hessian comparison theorem we can obtain the lower bound of $\lambda_{1,p}(B_o(R))$ where $B_o(R)$ is the geodesic ball in manifold with sectional curvature bounded above. Further details are provided in Corollary \ref{coro3.1}.
 \par Another application of Theorem \ref{thm1.2} is demonstrated in the following corollary:
\begin{corollary}\label{coro1.3}
 Let $\Omega$ be a bounded domain of an $n+1-$dimensional Riemannian manifold  satisfying that ${\rm Ric}[g]\geq-ng$ with  smooth  boundary $\partial\Omega.$ If the inscribed radius of $\Omega$ is $R$ and the mean curvature $H$ of $\partial\Omega$ with respect to outer normal satisfies that $H\geq k$ for some constant $k\geq n,$ then for any $p>1,$
   \begin{equation}\label{1.7}
     \lambda_{1,p}(\Omega)\geq
       \Big(\frac{k}{p}\Big)^p\Big(1+\frac{\pi^2}{(1+\frac{k}{p}R)^2}\Big)^{\min\{p-1,1\}}.
   \end{equation}
\end{corollary}
It is important to note that if $k>n,$ then the inscribed radius $R$ of $\Omega$ must satisfy that $R<\text{arccoth}\frac{k}{n}.$ Consequently, $\Omega$ is compact as long as $\partial\Omega$ is compact. For further details on this geometric property, readers are encouraged to consult in  \cite{kasue1983ricci},\cite{li2015rigidity}.
\par With the presentations above, we finally get the estimate of the $p-$eigenvalue of geodesic balls in AHE manifold (including the hyperbolic space).
\begin{theorem}\label{thm1.4}
  Let $(M,g)$ be a $C^{3,\alpha} \ (\alpha\in(0,1))\  n+1-$dimensional asymptotically hyperbolic Einstein manifold with conformal infinity of nonnegative Yamabe type, then for any $p>1,\ o\in M$ and as $R\rightarrow \infty,$
     \begin{equation}\label{1.8}
     \begin{aligned}
 &\big(\frac{n}{p}\big)^p+\big(\frac{n}{p}\big)^{p-2}\min\{p-1,1\}\frac{\pi^2}{R^2}+O(R^{-3})\leq \lambda_{1,p}(B_o(R)) \\ & \leq \Big(\frac{n}{p}\Big)^p+\Big(\frac{n}{p}\Big)^{p-2}\frac{p}{2}\frac{\pi^2}{R^2}+O(\frac{1}{R^{\min\{p+1,4\}}}).
     \end{aligned}
   \end{equation}
\end{theorem}
We need to expound that the asymptotical behavior in formula (\ref{1.8}) is
related to the geometry structure of manifolds. As it is showed in \cite{jin2024asymptotic} for the special case $p=2,$ the formula is related to the connectivity of manifolds at infinity.
\par To be more specific, let $(M,g)$ be an $n+1-$dimensional AHE manifold with conformal infinity $(\partial M, [\hat{g}])$, and it is shown in Witten–Yau \cite{witten1999connectedness}, or Cai-Galloway \cite{cai1999boundaries} or Theorem 5 in \cite{hijazi2020cheeger} that
$$Y(\partial M, [\hat{g}])\geq 0\Rightarrow \partial M\ \text{is connected ($M$ has one end)}.$$
Furthermore, combining the results of Lee \cite{lee1995spectrum} and Wang \cite{wang2001conformally}, we have:
$$Y(\partial M, [\hat{g}])\geq 0\Rightarrow\lambda_{1,2}(M)=\frac{n^2}{4}\Rightarrow \partial M\ \text{is connected ($M$ has one end)}.$$
Based on theorem 0.5 and 0.6 in \cite{li2002complete}, the author in \cite{jin2024asymptotic} removed the assumption that $(M,g)$ is AHE and obtained the following property:
$$\begin{cases}
 {\rm Ric} \geq -ng
 \\
 \lambda_{1,2}(B(o,R))=\frac{n^2}{4}+\frac{\pi^2}{R^2}+O(R^{-3})
 \end{cases}
 \Rightarrow  M\ \text{has one end.}$$
We now extend this property to the general $p-$case and the proof is mainly based on Theorem 1.4 in \cite{sung2014sharp}.
\begin{prop}\label{prop1.5}
  Let $(M,g)$ be a complete manifold of dimension $n+1\geq 3$ with ${\rm Ric}\geq-ng.$ If
     \begin{equation}\label{1.9}
\lambda_{1,p}(B_o(R))\geq\big(\frac{n}{p}\big)^p+\big(\frac{n}{p}\big)^{p-2}\min\{p-1,1\}\frac{\pi^2}{R^2}+O(R^{-3}),\ \ R\rightarrow\infty,
   \end{equation}
   for some $o\in M$ and some $p\in (\frac{8}{7},8)\bigcap (1,\frac{n^2}{2(n-1)}],$ then $M$ has only one end.
\end{prop}

Here is the outline of this paper: we use the Young's inequality to prove Theorem \ref{thm1.1} and as applications, we prove Proposition \ref{prop2.1} and Theorem \ref{thm2.2} and  Proposition \ref{prop2.5}  in Section \ref{sec2} and present additional applications of the theorems. In Section \ref{sec3}, we construct a new test function based on the distance function to prove Theorem \ref{thm1.2}. The techniques applied differ significantly for the cases where $p\in(1,2]$ and $p\in[2,+\infty).$ Additionally, Corollary \ref{coro1.3} is verified using classical techniques from Riemannian geometry. Section \ref{sec4} explores the concept of asymptotically hyperbolic Einstein (AHE) manifolds and concludes with a proof of Theorem \ref{thm1.4}. Finally, we prove Proposition \ref{prop1.5} in Section
\ref{sec5}.

\section{Estimates for lower bound of eigenvalues via functions}\label{sec2}
We will use the Young's inequality to prove Theorem \ref{thm1.1} in this section.
Under the conditions of Theorem \ref{thm1.1}, for any function $v\in C^\infty_0(\Omega),$ we have that
\begin{equation}\label{2.1}
\begin{aligned}
  D\int_\Omega |v|^p d\upsilon_g &\leq \int_\Omega |v|^p\cdot\Delta_{p_1} f d\upsilon_g -C\int_\Omega |v|^p\cdot|\nabla f|^{p_2}d\upsilon_g
  \\ &=-p\int_\Omega|v|^{p-1}\cdot g(\nabla|v|,|\nabla f|^{p_1-2}\nabla f) d\upsilon_g  -C\int_\Omega |v|^p\cdot|\nabla f|^{p_2}d\upsilon_g
  \\ &\leq p\int_\Omega|v|^{p-1}\cdot|\nabla v|\cdot|\nabla f|^{p_1-1}d\upsilon_g -C\int_\Omega |v|^p\cdot|\nabla f|^{p_2}d\upsilon_g
  \\ &\leq p\int_\Omega \Big[\frac{(|v|^{p-1}\cdot|\nabla f|^{p_1-1}\cdot\theta)^q}{q}+\frac{(\frac{|\nabla v|}{\theta})^p}{p}\Big]d\upsilon_g -C\int_\Omega |v|^p\cdot|\nabla f|^{p_2}d\upsilon_g
\end{aligned}
\end{equation}
Here $q$ is the conjugate
of $p$ and the last inequality holds because of Young's inequality. If we choose $\theta=\left(\frac{Cq}{p}\right)^{\frac{1}{q}}$ and notice that$(p_1-1)q=p_2,$ then we get
\begin{equation}\label{2.2}
  D\int_\Omega |v|^pd\upsilon_g \leq \frac{1}{\theta^p}\int_\Omega |\nabla v|^p d\upsilon_g .
\end{equation}
Since $C^\infty_0(\Omega)$ is dense in $W^{1,p}_0(\Omega),$ we get that
$$\lambda_{1,p}(\Omega)\geq\theta^p D=\left(\frac{C}{p-1}\right)^{p-1}\cdot D$$

\subsection{Some applications of Theorem \ref{thm1.1}}\label{sec2.1}
As the first application, we use Theorem \ref{thm1.1} to obtain a Barta's inequality for $p-$Laplace operator:
 \begin{prop}[Theorem 2.1 in \cite{allegretto1998picone}]\label{prop2.1}
 Let $\Omega$ be a bounded smooth domain in a Riemannian manifold.  If there exists a positive function $v\in W^{1,p}(\Omega)$ satisfying that
 $\Delta_{p}v\leq -\mu v^{p-1}$ in $\Omega$ for some constant $\mu.$ Then
$
  \lambda_{1,p}(\Omega)\geq \mu.
$
\end{prop}

\begin{proof}
For any $f\in W^{1,p}(\Omega),$
\begin{equation}\label{2.3}
  \begin{aligned}
    -\Delta_p e^{-f}&=-\text{div}(|\nabla e^{-f}|^{p-2}\nabla e^{-f})=\text{div}(e^{-(p-1)f}|\nabla f|^{p-2}\nabla f)
    \\&=e^{-(p-1)f}\text{div}(|\nabla f|^{p-2}\nabla f)+g(\nabla e^{-(p-1)f},|\nabla f|^{p-2}\nabla f)
    \\&=e^{-(p-1)f}\Delta_p f-(p-1)e^{-(p-1)f}|\nabla f|^{p}
    \\&=(e^{-f})^{p-1}\cdot[\Delta_p f-(p-1)|\nabla f|^{p}]
  \end{aligned}
\end{equation}
If we set $f=-\ln v,$ then
$$
 \Delta_p f-(p-1)|\nabla f|^{p}=-\frac{\Delta_p v}{v^{p-1}}\geq\mu
$$
Then $\lambda_{1,p}(\Omega)\geq \mu$ by Theorem \ref{thm1.1}.
\end{proof}

\par Another application of Theorem \ref{thm1.1} is illustrated in the following proposition:
 \begin{theorem}\label{thm2.2}
   Let $\Omega$ be a bounded smooth domain in a Riemannian manifold
and $f\in W^{1,q}(\Omega)$ satisfying that
$|\nabla f|\leq a$ and $\Delta_{q} f\geq b$ for some constants $a, b > 0$ and $q>1.$ Then for any $p>1,$
$$
  \lambda_{1,p}(\Omega)\geq\frac{b^p}{p^pa^{p(q-1)}}.
$$
 \end{theorem}
 This result extends Theorem 1.1 in \cite{carvalho2022fundamental}, which was initially discussed in a special case where $p=q.$ Theorem \ref{thm2.2} broadens the applicability as it does not restrict the relationship between $p$ and $q.$

 \begin{proof}
 Suppose that $f,q,b,a$ are defined as in Theorem \ref{thm2.2}. For any $p>1,$ set
$$
p_2=\frac{p(q-1)}{p-1},\ \ \ C=\frac{p-1}{p}\cdot\frac{b}{a^{p_2}},\ \ \ D=\frac{b}{p}.
$$
 Then $p=\frac{p_2}{p_2-q+1}$ and
$$
\Delta_{q}f-C|\nabla f|^{p_2}\geq b-\frac{p-1}{p}\cdot\frac{b}{a^{p_2}}\cdot a^{p_2}=D.
$$
Thus, the first eigenvalue $\lambda_{1,p}(\Omega)$ is bounded below by:
$$
\lambda_{1,p}(\Omega)\geq\Big(\frac{C}{p-1}\Big)^{p-1}\cdot D=\frac{b^p}{p^pa^{p(q-1)}}.
$$
\end{proof}
\par Theorem \ref{thm2.2} also leads to interesting applications when combined with gradient estimates for the $p-$Laplacian equations. For instance,
Theorem 1.1 in \cite{sung2014sharp}  suggests the following proposition:
\begin{corollary}\label{coro2.3}
Let $(M,g)$ be an $n+1-$dimensional complete noncompact
manifold with ${\rm Ric}[g]\geq-ng.$ Suppose $q>1$ and there exists a positive solution to the equation
$\Delta_{q} v=-\lambda_{1,q}(M)v^{q-1}.$ Then for any $p>1.$
\begin{equation}\label{2.4}
  \lambda_{1,p}(M)\geq \left(\frac{\lambda_{1,q}(M)}{y^{q-1}p}\right)^p
\end{equation}
where $y$ is the positive root of the equation
$$
  (p-1)y^p-ny^{p-1}+\lambda_{1,q}(M)=0.
$$
Moreover, if $\lambda_{1,q}(M)=\big(\frac{n}{q}\big)^{q},$ then for any $p>1,$
\begin{equation}\label{2.5}
  \lambda_{1,p}(M)
  \begin{cases}
    \geq\big(\frac{n}{qp}\big)^{p}, & {\rm if}\ p\in(1,q) \\
    =\big(\frac{n}{p}\big)^{p}, &  {\rm if}\ p\in[q,+\infty).
  \end{cases}
\end{equation}
\end{corollary}
\begin{proof}
Recall that $\lambda_{1,p}(M)$ is defined as the limit of $\lambda_{1,p}(\Omega_k)$ for any smoothly compact exhaustion $\{\Omega_k\}_{k=1}^\infty$ of $M.$ We find that Theorem \ref{thm1.1} is also applicable to $M,$ as long as the condition ``$f\in W^{1,p_1}(\Omega)$" is modified to
``$f\in W^{1,p_1}_{\text{loc}}(M).$" This modification is equally valid for Proposition \ref{prop2.1} and Theorem \ref{thm2.2}.  Consider setting
$f=-\ln v,$ then according to (\ref{2.3}), we derive:
$$
  \Delta_{q}f=(q-1)|\nabla f|^{q}+\lambda_{1,q}(M)
$$
implying that $\Delta_{q}f\geq \lambda_{1,q}(M).$ On the other hand, Theorem 1.1 in \cite{sung2014sharp} provides an
estimate for the gradient, denoting $|\nabla f|\leq y$ where $y$ is defined as in this corollary. Consequently, this leads to equation
(\ref{2.4}) by Theorem \ref{thm2.2}.
\par If $\lambda_{1,q}(M)=\big(\frac{n}{q}\big)^{q},$ then $y=\frac{n}{q}$ and hence (\ref{2.5}) holds for $p\in(1,q).$ When $p\geq q,$ then on the one hand, $\lambda_{1,p}(M)\leq\big(\frac{n}{p}\big)^{p}$ by the Cheng type inequality (Theorem 2 in \cite{takeuchi1998first}). On the other hand, $\lambda_{1,p}(M)\geq\big(\frac{n}{p}\big)^{p}$ due to the following lemma by setting $\Omega\rightarrow M.$
\end{proof}
\begin{lemma}[An extension of Cheeger inequality: (4.3.5) in \cite{maz2011sobolev}]
Let $\Omega$ be a bounded smooth domain in a Riemannian manifold. Then the function
$$p\mapsto p\cdot(\lambda_{1,p}(\Omega))^{\frac{1}{p}}$$
is increasing for $p>1.$
\end{lemma}
\begin{proof}
  For any $1<q<p$ and any $u\in C^\infty_0(\Omega),$
$$
  \begin{aligned}
    \lambda_{1,q}(\Omega)&\leq \frac{\int_\Omega |\nabla(u^\frac{p}{q})|^q d\upsilon_g}{\int_\Omega (u^\frac{p}{q})^q d\upsilon_g}=\Big(\frac{p}{q}\Big)^q\frac{\int_\Omega u^{p-q}|\nabla u|^q d\upsilon_g}{\int_\Omega u^p d\upsilon_g}
    \\ &\leq \Big(\frac{p}{q}\Big)^q\frac{\Big[\int_\Omega (u^{p-q})^{\frac{p}{p-q}} d\upsilon_g\Big]^{\frac{p-q}{p}} \Big[\int_\Omega(|\nabla u|^q)^{\frac{p}{q}} d\upsilon_g\Big]^{\frac{q}{p}}}{\int_\Omega u^p d\upsilon_g}
    \\ &=\Big(\frac{p}{q}\Big)^q\left(\frac{\int_\Omega |\nabla u|^p d\upsilon_g }{\int_\Omega u^p d\upsilon_g}\right)^{\frac{q}{p}}
    \end{aligned}
$$
Since $C^\infty_0(\Omega)$ is dense in $W^{1,p}_0(\Omega),$ we have
  $$\lambda_{1,q}(\Omega)\leq\Big(\frac{p}{q}\Big)^q\cdot(\lambda_{1,q}(\Omega))^{\frac{q}{p}}$$
  and this is equivalent to
   $$q\cdot(\lambda_{1,q}(\Omega))^{\frac{1}{q}}\leq p\cdot(\lambda_{1,p}(\Omega))^{\frac{1}{p}}$$
\end{proof}

\par  Now let us introduce the third application of Theorem \ref{thm1.1}. Recall that the authors utilize Theorem 1.1 in \cite{carvalho2022fundamental} to provide a straightforward proof for the generalization of McKean's theorem, asserting that if $M$ is an $n+1-$dimensional complete simply connected Riemannian manifold such that the sectional curvature is bounded above by $-1,$ then
$\lambda_{1,p}(M)\geq\Big(\frac{n}{p}\Big)^p.$ This was previously mentioned by Poliquin \cite{poliquin2014bounds} using estimates by the Cheeger constant and the Cheeger type inequality for $p-$Laplacian by Theorem 2 in \cite{takeuchi1998first}.
\par This result was further extended to the asymptotically hyperbolic Einstein (AHE) manifold in \cite{hijazi2020cheeger},
which demonstrated that the Cheeger constant of an $(n+1)$-dimensional AHE manifold $M$ with nonnegative Yamabe type conformal infinity equals to $n.$ As a consequence, they proved that

 \begin{prop}[Theorem 9 in \cite{hijazi2020cheeger}]\label{prop2.5}
 Let $(M,g)$ be an $n+1-$dimensional asymptotically hyperbolic Einstein manifold with conformal infinity of nonnegative Yamabe type. Then for any $p>1,$
  \begin{equation}\label{2.6}\lambda_{1,p}(M)=\Big(\frac{n}{p}\Big)^p.
  \end{equation}
\end{prop}
This is a generalization of the classical Lee's spectral estimate $(p=2)$ in \cite{lee1995spectrum}.
\begin{proof}[New proof]
Let $u$ be the eigenfunction solution to $\Delta u=(n+1)u$ which was first introduced by Lee in \cite{lee1995spectrum} and set $f=\ln u,$ then a direct calculation indicates that
$$\Delta f\geq n\ \ \text{and}\ \ |\nabla f|\leq 1.$$
 Then $\lambda_{1,p}(M)\geq(\frac{n}{p})^p$ by letting $\Omega\rightarrow M$ in Theorem \ref{thm2.2}. Hence $\lambda_{1,p}(M)=(\frac{n}{p})^p$ according to the Cheng type inequality (Theorem 2 in \cite{takeuchi1998first}).
 \end{proof}

\section{The estimate of eigenvalue for domain of bounded ``radius"}\label{sec3}
We will first prove Theorem \ref{thm1.2} in this section. Here is the main idea of the proof: we utilize the distance function $r$ to construct a new subsolution $f$ on $\Omega$
such that condition (\ref{1.4}) is satisfied for certain $p_1,p_2,C,D.$ When $p\in (1,2],$ we set $p_1=p_2=p$ and this method is essentially equivalent to the Barta' inequality, i.e. Proposition \ref{prop2.1}.  When $p\in[2,+\infty),$ we set $p_1=2$ and $ p_2=\frac{1}{p-1}+1\leq 2$ and apply Theorem \ref{thm1.1} to achieve the desired results.
 \par Let $r:\Omega\rightarrow[0,R]$ be the distance function satisfying that $|dr|\equiv 1$ and $\Delta r\geq k>0$ almost everywhere. Set
 \begin{equation}\label{3.1}
 f(\cdot)=r(\cdot)-\frac{p}{k}\ln \sin a\big(r(\cdot)+\frac{p}{k}\big)
 \end{equation}
 where $a=\frac{\pi-\varepsilon}{R+\frac{p}{k}}>0$ is a constant. Here $\varepsilon<\pi$ is a small positive number. The derivative of $f$ is given by:
\begin{equation}\label{3.2}
  \dot{f}(r)=1-\frac{p}{k}a\cot a(r+\frac{p}{k}).
\end{equation}
Then we obtain that
$$\dot{f}(r)\geq \dot{f}(0)=1-\frac{p}{k}a\cot\frac{p}{k}a>0.$$ The second derivative of $f$ is
\begin{equation}\label{3.3}
  \ddot{f}=\frac{p}{k}a^2\frac{1}{\sin^2 a(r+\frac{p}{k})}=\frac{k}{p}\Big((1-\dot{f})^2+\Big(\frac{p}{k}\Big)^2a^2\Big)
\end{equation}
Given that $\dot{f}>0$ and $|\nabla r|\equiv 1,$ we have
\begin{equation}\label{3.4}
\begin{aligned}
  \Delta_p f&=\text{div}(|\nabla f|^{p-2}\nabla f)=\text{div}(|\dot{f}\nabla r|^{p-2}\dot{f}\nabla r)
  \\ &=\text{div}(\dot{f}^{p-1}\nabla r)=\dot{f}^{p-1}\Delta r+g(\nabla\dot{f}^{p-1},\nabla r)
  \\ &=\dot{f}^{p-1}\Delta r+(p-1)\dot{f}^{p-2}\ddot{f}.
\end{aligned}
\end{equation}

\textbf{Case 1, $p\in(1,2].$}
\par We choose $p_1=p_2=p$ and $C=(p-1)\frac{k}{p},$ then
\begin{equation}\label{3.5}
\begin{aligned}
 \Delta_p f-C|\nabla f|^p&\geq k\dot{f}^{p-1}+(p-1)\dot{f}^{p-2}\ddot{f}-(p-1)\frac{k}{p}\dot{f}^p
 \\ &=\frac{k}{p}\dot{f}^{p-2}\Big[p\dot{f}+(p-1)\Big((1-\dot{f})^2+\Big(\frac{p}{k}\Big)^2a^2\Big)-(p-1)\dot{f}^2\Big]
 \\ &=\frac{k}{p}\dot{f}^{p-2}\Big[(2-p)\dot{f}+(p-1)\Big(1+\frac{p^2}{k^2}a^2\Big)\Big]
\end{aligned}
\end{equation}
By defining $m=(p-1)(1+\frac{p^2}{k^2}a^2),$ we observe that the function
$$h(x)=x^{p-2}[(2-p)x+m] \ \ \ \ x\in (0,1+\frac{p}{k})$$
 achieves its minimum at $x=\frac{m}{p-1}.$ Hence
\begin{equation}\label{3.6}
\begin{aligned}
 \Delta_p f-C|\nabla f|^p&\geq \frac{k}{p} h(\frac{m}{p-1})
 =\frac{k}{p}\Big(\frac{m}{p-1}\Big)^{p-2}\Big[(2-p)\frac{m}{p-1}+m\Big]
 \\ &=\frac{k}{p}\Big(\frac{m}{p-1}\Big)^{p-1}=\frac{k}{p}\Big(1+\frac{p^2}{k^2}a^2\Big)^{p-1}=D
\end{aligned}
\end{equation}
Thus by Theorem \ref{thm1.1}, we conclude that
\begin{equation}\label{3.7}
  \lambda_{1,p}(\Omega)\geq\Big(\frac{C}{p-1}\Big)^{p-1}\cdot D=\Big(\frac{k}{p}\Big)^{p}\Big(1+\frac{p^2}{k^2}a^2\Big)^{p-1}
\end{equation}
Let $\varepsilon\rightarrow 0,$ we complete the proof for equation (\ref{1.6}) for $p\in (1,2].$
\\
\par
\textbf{Case 2,  $p\in [2,+\infty).$}
\par We choose $p_1=2, p_2=\frac{1}{p-1}+1$ and $C=(p-1)\frac{k}{p},$ then
\begin{equation}\label{3.8}
\begin{aligned}
 \Delta f-C|\nabla f|^{p_2}&\geq k\dot{f}+\ddot{f}-(p-1)\frac{k}{p}\dot{f}^{p_2}
 \\ &=\frac{k}{p}\Big[p\dot{f}+(1-\dot{f})^2+\Big(\frac{p}{k}\Big)^2a^2-(p-1)\dot{f}^{p_2}\Big]
 \\ &=\frac{k}{p}\Big(\dot{f}^2+(p-2)\dot{f}-(p-1)\dot{f}^{\frac{1}{p-1}+1}+1+\frac{p^2}{k^2}a^2\Big)
\end{aligned}
\end{equation}
We define
 $$h(x)=x^2+(p-2)x-(p-1)x^{\frac{1}{p-1}+1}\ \ \ \ x\in(0,+\infty).$$
Then $h(x)\geq 0$ is equivalent to
$$ x+(p-2)-(p-1)x^{\frac{1}{p-1}}\geq 0$$
which is obviously since $p\geq 2.$ Back to the equation (\ref{3.8}),
\begin{equation}\label{3.9}
 \Delta f-C|\nabla f|^{p_2}\geq \frac{k}{p}\Big(1+\frac{p^2}{k^2}a^2\Big).
\end{equation}
Therefore,
 \begin{equation}\label{3.10}
  \lambda_{1,p}(\Omega)\geq\Big(\frac{C}{p-1}\Big)^{p-1}\cdot  \frac{k}{p}\Big(1+\frac{p^2}{k^2}a^2\Big)=\Big(\frac{k}{p}\Big)^{p}\Big(1+\frac{p^2}{k^2}a^2\Big)
\end{equation}
With $\varepsilon\rightarrow 0,$  we complete the proof for Theorem \ref{thm1.2}.
\\
\par Finally, we present a direct application of Theorem \ref{thm1.2}:
\begin{corollary}\label{coro3.1}
  Assume that $(M,g)$ is a complete Riemannian manifold of dimension $n+1$ whose sectional curvature satisfies that $K_M\leq-\kappa^2$ for some $\kappa>0.$ Then for any $o\in M$ and any $R>0,$
   \begin{equation}\label{3.11}
    \lambda_{1,p}(B_o(R))\geq
    \big(\frac{n\kappa}{p}\big)^p\coth^p(\kappa R)\big[1+\frac{\pi^2}{(1+\frac{n\kappa}{p}R\coth(\kappa R))^2}\big]^{\min\{p-1,1\}}.
   \end{equation}
   In special, if $R$ is large, then
  \begin{equation}\label{3.12}
    \lambda_{1,p}(B_o(R))\geq
       \big(\frac{n\kappa}{p}\big)^p+ \big(\frac{n\kappa}{p}\big)^{p-2}\min\{p-1,1\}\frac{\pi^2}{R^2}+O(R^{-3}).
    \end{equation}
\end{corollary}
\begin{proof}
 We consider the distance function $r={\rm dist}(o,\cdot)$ in $B_o(R),$ then
$$\Delta r\geq n\kappa\coth \kappa R>n\kappa$$
in the sense of distribution by the Hessian comparison theorem. Then Theorem \ref{thm1.2} would imply (\ref{3.11}) and (\ref{3.12}).
\end{proof}

\subsection{The estimate of eigenvalue for bounded domain with bounded Ricci and mean curvature}
In this subsection, we will demonstrate the proof of Corollary \ref{coro1.3}. This proof requires us to estimate the mean curvature of the level sets of the distance function from the boundary, employing methods standard to Riemannian geometry, similar to those described in Proposition 2 in \cite{cai1999boundaries}.
\par
Assume that $\Omega$ is a bounded domain in an $n+1-$dimensionla manifold $(M,g)$ whose Ricci curvature is bounded from below, i.e.
${\rm Ric}[g]\geq-n g.$ Suppose further that $\partial\Omega$ is smooth and the mean curvature at the boundary $H|_{\partial\Omega}\geq k\geq n.$ We define the distance function
\begin{equation}\label{3.13}
\rho(x)=\text{dist}(x,\partial\Omega):\ \ \ \overline{\Omega}\rightarrow [0,R]
\end{equation}
where $R=\sup\limits_{x\in\Omega} \rho(x)$ represents  the inscribed radius of $\Omega.$ The function $\rho$ is smooth in $\Omega$ outside the cut locus which is a set of zero measure. For a given point $q\in\partial\Omega,$ we set
$\sigma:[0,T)\rightarrow \Omega$ to be the normal geodesic satisfying $\sigma(0)=q$ and $\dot{\sigma}(0)\perp T_q\partial\Omega.$ Here $\sigma(T)$ is the focal point and $T\leq R.$ Then $\rho\circ\sigma$ is continuous in $[0,T]$ and smooth in $(0,T).$ Let
$$H(s)=-\Delta\rho|_{\sigma(s)}$$
be the mean curvature of the level set $\{\rho=s\}$ with respect
to the outer normal $-\nabla\rho$ at $\sigma(s)$ and $H(0)=H|_q\geq k.$ According to the Riccati equation:
\begin{equation}\label{3.14}
  H'(s)=|\text{Hess}\rho(\sigma(s))|^2+{\rm Ric}(\dot{\sigma}(s),\dot{\sigma}(s))
\end{equation}
Defining  $h(s)=\frac{H(s)}{n},$ then $h$ satisfies
$$h'(s)\geq h^2(s)-1,\ \ h(0)=\frac{H(0)}{n}.$$
Let $y(s)$ be the unique solution to
$$y'(s)=y^2(s)-1,\ \ y(0)=\frac{H(0)}{n}.$$
 In fact,
if $H(0)=n,$ then $y(s)=1$ for $s\in[0,T].$ If $H(0)>n,$ then
$$y(s)=\coth(-s+\text{arccoth}(y(0)))$$ for $s\in[0,T)$ and $T<\text{arccoth}(y(0)).$ Consequently, we always have that
\begin{equation}\label{3.15}
y(s)\geq y(0)\geq \frac{k}{n}.
\end{equation}
Now consider the function $(h-y)e^{-\int((h+y)},$ which satisfies:
$$[(h-y)e^{-\int((h+y)}]'=e^{-\int((h+y)}(h'-y'-(h-y)(h+y))\geq 0.$$
This ensures that $h(s)\geq y(s)$ for all $s$ which implies $-\Delta\rho\geq k$ almost everywhere in $\Omega.$
In the end, we define the function $r=R-\rho.$ Then  $r$ satisfies the conditions of Theorem \ref{thm1.2}. Thus, we conclude the proof of the theorem.

\section{The first eigenvalue on AHE manifold}\label{sec4}
Firstly, we will introduce some basic materials about asymptotically hyperbolic manifold. Suppose that $\overline{M}$ is $n+1-$dimensional manifold with smooth boundary $\partial M$ of dimension $n.$ Let $M$ be its interior. A complete noncompact metric $g$ in $M$ is called smoothly ($C^{m,\alpha}$ or $W^{k,p}$)
conformally compact if there exists a defining function $\rho$ in $\overline{M}$ such that the conformal metric $\bar{g}=\rho^2g$ can extend to a smooth ($C^{m,\alpha}$ or $W^{k,p}$) Riemannian metric on $\overline{M}.$ Here the defining function $\rho$ satisfies
\begin{equation}\label{4.1}
  \rho>0\ \ \text{in} \ \  M, \ \ \ \ \rho=0\ \ \text{on} \ \partial M,\ \ \ \ \ d\rho\neq 0 \ \ \text{on} \  \ \partial M.
\end{equation}
We call $\hat{g}=\bar{g}|_{T\partial M}$ the boundary metric associated to the compactification $\bar{g}.$ It is well known that $(M,g)$ induces a conformal structure $(\partial M,[\hat{g}])$ and we call it the conformal infinity of $(M,g).$
\par Let $(M,g)$ be a conformally compact manifold and $\bar{g}=\rho^2g$ be a $C^2$ compactification. A straightforward calculation indicates that
the curvature of $(M,g)$ is of the following from \cite{graham1999volume}:
\begin{equation}\label{4.2}
  R_{ijkl}[g]=|d\rho|^2_{\bar{g}}(g_{ik}g_{jl}-g_{il}g_{jk})+O_{ijkl}(\rho^{-3})
\end{equation}
near $\partial M.$ As a consequence, the sectional curvature $K[g]=-|d\rho|^2_{\bar{g}}+O(\rho)$ is uniformly approaching to $-|d\rho|^2_{\bar{g}}$ (see \cite{mazzeo1988hodge}). Thus if in addition $|d\rho|^2_{\bar{g}}|_{\partial M}=1,$ we say $(M,g)$ is an asymptotically hyperbolic manifold or AH manifold for short.
\par Let $(M,g)$ be a $C^2$ conformally compact manifold. If $g$ is also Einstein:
${\rm Ric}[g]=-ng.$ Then a direct calculation yields that $|d\rho|^2_{\rho^2g}|_{\partial M}=1,$ and hence we call $(M,g)$ an  asymptotically hyperbolic Einstein manifold or AHE manifold for short.
\par Suppose that $(M,g)$ is a $C^{3,\alpha}$ AH manifold and $\hat{g}\in [\hat{g}]$ is a boundary representative, then there exists a unique defining function $x$ such that
$|dx|_{x^2g}\equiv 1$ in a neighbourhood of $\partial M$ and $x^2g|_{T\partial M}=\hat{g},$ \cite{graham1991einstein}\cite{lee1995spectrum}. We say that $x$ is the geodesic defining function associated with $\hat{g}$ and $\bar{g}=x^2g$ is the ($C^{2,\alpha}$) geodesic conformal compactification. In this case, the function $x$ determines an identification of $\partial M\times [0, \delta)$ in a neighbourhood of $\partial M$ in $\overline{M}$ for some small $\delta > 0.$
\subsection{Lower bound estimate}
We denote
$M_\varepsilon=\partial M\times(0,\varepsilon)$ and $E_\varepsilon=M\setminus M_\varepsilon$ for $\varepsilon<\delta.$
\begin{lemma}\label{lemma4.1}
Let $(M,g)$ be a  $C^{3,\alpha}$ AHE manifold and  $\bar{g}=x^2g$ is the geodesic conformally compactification with boundary metric $\hat{g}=\bar{g}|_{T\partial M}.$ If the scalar curvature $S_{\hat{g}}\geq0,$  then for $\varepsilon>0$ sufficiently small, the mean curvature $H_\varepsilon$ of $\partial E_\varepsilon$ in $(E_\varepsilon,g)$ with respect to the outer normal satisfies that
$H_\varepsilon\geq n.$
\end{lemma}
If the scalar curvature $S_{\hat{g}}>0,$ then in \cite{wang2002new}, Wang proved that $H_\varepsilon \geq n+c\varepsilon^2$ where
$c>0$ by a direct calculation. If $S_{\hat{g}}\geq 0,$ then in the following we provide a concise proof to simplify the concepts for readers, although the property is considered trivial in the study of asymptotically hyperbolic Einstein (AHE) metrics. For a comprehensive understanding, one may refer to the Appendix in \cite{anderson2008einstein}.
\begin{proof}
We use ${\rm \bar{R}ic}$ and $\bar{S}$ to denote Ricci curvature and scalar curvature of $\bar{g}$  and $\bar{D}^2$ denote the Hessian of $\bar{g}$. Then
by the  conformal transformation law of curvatures in \cite{besse2007einstein},
\begin{equation}\label{4.3}
{\rm \bar{R}ic}=-(n-1)\frac{\bar{D}^2x}{x}-\frac{\Delta_{\bar{g}} x}{x}\bar{g},
\end{equation}
\begin{equation}\label{4.4}
\bar{S}=-2n\frac{\Delta_{\bar{g}}x}{x}
\end{equation}
Let $$\bar{H}(x)=\Delta_{\bar{g}} x=-\frac{1}{2n}x\bar{S}$$ be the mean curvature of the level set of $x$ in $(\overline{M},\bar{g}),$ then the
Riccati equation indicates that
\begin{equation}\label{4.5}
\bar{H}'(x)+|\bar{D}^2x|^2+{\rm \bar{R}ic}(dx,dx)=0.
\end{equation}
Hence
\begin{equation}\label{4.6}
\bar{S}'(x)=2n\frac{|\bar{D}^2x|^2}{x}\geq 2n\frac{(\Delta_{\bar{g}} x)^2}{nx}=\frac{x}{2n^2}\bar{S}^2\geq 0.
\end{equation}
On the other hand, from (A.8) in \cite{anderson2008einstein} or Lemma 3.2 in \cite{jin2021finite}, we have that
$$\bar{S}(0)=\bar{S}|_{\partial X}=\frac{n}{n-1}S_{\hat{g}}\geq 0.$$
 Then $\bar{S}\geq 0$ in $\partial X\times [0, \delta).$ In the end, as $-\ln x$ is the distance function of $(E_\varepsilon, g),$ we can obtain that
\begin{equation}\label{4.7}
H_\varepsilon=\Delta_g(-\ln x)|_{\{x=\varepsilon\}}=n+\frac{\bar{S}|_{\{x=\varepsilon\}}}{2n}\varepsilon^2\geq n.
\end{equation}
This concludes the proof.
\end{proof}
\begin{proof}[Proof of Theorem \ref{thm1.4}]
 Suppose that $(M,g)$ is an AHE manifold with conformal infinity $(\partial M,[\hat{g}])$ of nonnegative Yamabe type. Hence we
could choose a representative boundary metric $\hat{g}\in [\hat{g}]$ such that the scalar curvature of $\hat{g}$ satisfies that $S_{\hat{g}}\geq 0.$
Let $x$ be the geodesic defining function associated to $\hat{g},$ which, as indicated by Lemma \ref{lemma4.1}, ensures $H_\varepsilon\geq n$ for small $\varepsilon>0.$
\par For any fixed $o\in M,$  select a small $\varepsilon_1>0$ such that $o\in E_{\varepsilon_1}.$ Assume that
$d=\text{dist}_g(o,\partial E_{\varepsilon_1})\leq R_1$ where $R_1$ is the inscribed radius of $E_{\varepsilon_1}.$ Then for any $R>0,$ set
$\varepsilon=\varepsilon_1e^{-R}.$ We have that
$\text{dist}_g(o,\partial E_\varepsilon)=d+R$
and the inscribed radius of $E_\varepsilon$ is $R_1+R.$(One can see (2.16) in \cite{jin2024rigidity} for more details). Therefore, by  the domain monotonicity and Theorem \ref{thm1.2}
\begin{equation}\label{4.8}
  \lambda_{1,p}(B_o(d+R))\geq \lambda_{1,p}(E_\varepsilon)\geq
   \big(\frac{n}{p}\big)^p\Big[1+\frac{\pi^2}{(1+\frac{n}{p}(R_1+R))^2}\Big]^{\min{(p-1,1)}}.
\end{equation}
Let $R\rightarrow+\infty,$ we establish the lower bound as stated in Theorem \ref{thm1.4}.
\end{proof}
\subsection{Upper bound estimate}
To derive the asymptotic expansion of $\lambda_{1,p}(B_o(R))$ on AHE manifold (i.e. Theorem \ref{thm1.4}), we briefly discuss upper bounds for the $p$-Laplacian, even though the primary focus of this paper is on lower bounds. For further results on $\lambda_{1,p}$ in asymptotically hyperbolic (AH) manifolds, we refer the reader to \cite{perez2024first}.
\begin{lemma}\label{lemma4.2}
Assume that
$(M,g)$ is an $n+1-$ dimensional complete non-compact manifold satisfying that ${\rm Ric}[g]\geq-ng,$ then for any $p>1,o\in M,$
\begin{equation}\label{4.9}
 \lambda_{1,p}(B_o(R))\leq \left(\frac{n}{p}\right)^p+\left(\frac{n}{p}\right)^{p-2}\frac{p}{2}\cdot\frac{\pi^2}{R^2}+O(R^{-1-p})+O(R^{-4}) \ \ \ R\rightarrow+\infty
\end{equation}
\end{lemma}
\begin{proof}
Let us recall the classic eigenvalue comparison theorem of Cheng in \cite{cheng1975eigenvalue} and \cite{takeuchi1998first}: for any $R>0,$
$$\lambda_{1,p}(B_o(R))\leq\lambda_{1,p} (B^\mathbb{H}(R)).$$
Here $B^\mathbb{H}(R)$ is a geodesic ball of radius $R$ in hyperbolic space $\mathbb{H}^{n+1}.$
 Hence we only need to make estimates of the upper bound first Dirichlet eigenvalue of $p-$Laplacian of geodesic balls in hyperbolic space.
 In the followings, we use
 $$
 \Gamma(p)=\int_0^\infty t^{p-1} e^{-t}\, dt$$
  $$B(p,q)=\int_0^1 t^{p-1}(1-t)^{q-1}\, dt=2\int_0^{\frac{\pi}{2}}\sin^{2p-1}\theta\cos^{2q-1}\theta\,d\theta
 $$
 to denote the Gamma function and Beta function separately.
\par Assume that $r$ is the distance function of the centre point.
Let $R$ be a large number and we consider the function  
$$f(\cdot)=e^{-\frac{n}{p}r(\cdot)}\sin\frac{\pi}{R}r(\cdot) \ \ \text{in} \ \ B^\mathbb{H}(R).$$ Then
\begin{equation}\label{4.10}
\begin{aligned}
\lambda_{1,p}(B^\mathbb{H}(R))&\leq \frac{\displaystyle\int_{B^\mathbb{H}(R)}|\nabla f|^p\ d\upsilon_g}{\displaystyle\int_{B^\mathbb{H}(R)}|f|^p d\upsilon_g}
\\ &=\frac{\displaystyle\int_0^R e^{-nr}\Big|-\frac{n}{p}\sin\frac{\pi r}{R}+ \frac{\pi}{R}\cos\frac{\pi r}{R}\Big|^p\omega_n\sinh^nr dr}{\displaystyle\int_0^R e^{-nr}\sin^p\frac{\pi r}{R}\omega_n\sinh^nr dr}
\\ &=\frac{\displaystyle\int_0^R (1-e^{-2r})^n\Big|-\frac{n}{p}\sin\frac{\pi r}{R}+ \frac{\pi}{R}\cos\frac{\pi r}{R}\Big|^p dr}{\displaystyle\int_0^R (1-e^{-2r})^n\sin^p\frac{\pi r}{R} dr}
\\ &=\frac{\displaystyle\int_0^\pi (1-e^{-\frac{2R}{\pi}\theta})^n\Big|-\frac{n}{p}\sin\theta+ \frac{\pi}{R}\cos\theta\Big|^p d\theta}{\displaystyle\int_0^\pi (1-e^{-\frac{2R}{\pi}\theta})^n\sin^p\theta d\theta}
\\ &=\Big(\frac{n^2}{p^2}+\frac{\pi^2}{R^2}\Big)^{\frac{p}{2}}\cdot
\frac{\displaystyle\int_0^\pi (1-e^{-\frac{2R}{\pi}\theta})^n\cdot|\sin(\theta-\alpha)|^p d\theta}{\displaystyle\int_0^\pi (1-e^{-\frac{2R}{\pi}\theta})^n\sin^p\theta d\theta}
\\&=\Big(\frac{n^2}{p^2}+\frac{\pi^2}{R^2}\Big)^{\frac{p}{2}}\cdot\frac{F(R)}{G(R)}
\end{aligned}
\end{equation}
Here $\alpha$ is a constant determined by
$$\sin \alpha=\frac{\pi}{R}\Big(\frac{n^2}{p^2}+\frac{\pi^2}{R^2}\Big)^{-\frac{1}{2}},\ \ \cos\alpha=\frac{n}{p}\Big(\frac{n^2}{p^2}+\frac{\pi^2}{R^2}\Big)^{-\frac{1}{2}}.$$
On the one hand,
\begin{equation}\label{4.11}
F(R)\leq  \int_0^\pi |\sin (\theta-\alpha)|^p d\theta=\frac{1}{2}B(\frac{p+1}{2},\frac{1}{2}).
\end{equation}
On the other hand,
for $R$ big enough, we have that
\begin{equation}\label{4.12}
\begin{aligned}
\int_0^\pi e^{-R\theta}\sin^p\theta d\theta &=\frac{1}{R}\int_0^{R\pi}e^{-t}\sin^p\frac{t}{R}dt
\\ & \leq \frac{1}{R}\int_0^{R\pi}e^{-t}\left(\frac{t}{R}\right)^pdt
\\ &\leq \frac{1}{R^{1+p}}\Gamma(p+1)
\end{aligned}
\end{equation}
 As a consequence,
\begin{equation}\label{4.13}
\begin{aligned}
  G(R)&=\int_0^\pi \sin^p\theta d\theta +\sum\limits_{k=1}^nC_n^k\int_0^\pi (-e^{-\frac{2R}{\pi}\theta})^k\sin^p\theta d\theta
  \\ &\geq \frac{1}{2}B(\frac{p+1}{2},\frac{1}{2})-\frac{C(n,p)}{R^{1+p}}
  \end{aligned}
\end{equation}
where $C(n,p)$ is a constant depending on $n$ and $p.$ Then (\ref{4.11}) and (\ref{4.13}) imply that
\begin{equation}\label{4.14}
\begin{aligned}
  \lambda_{1,p}(B^\mathbb{H}(R))&\leq\Big(\frac{n^2}{p^2}+\frac{\pi^2}{R^2}\Big)^{\frac{p}{2}}\frac{F(R)}{G(R)}
  \\   &\leq \Big(\frac{n}{p}\Big)^{p}\Big(1+\frac{p}{2}\frac{p^2}{n^2}\frac{\pi^2}{R^2}+O(R^{-4})\Big)\cdot(1+O(R^{-1-p}))
  \\ &=\Big(\frac{n}{p}\Big)^p+\Big(\frac{n}{p}\Big)^{p-2}\frac{p}{2}\cdot\frac{\pi^2}{R^2}+O(R^{-1-p})+O(R^{-4})
\end{aligned}
\end{equation}
\end{proof}

\section{The geometric property of the asymptotical behavior}\label{sec5}
In order to prove Proposition \ref{prop1.5}, we first introduce a result of Sung-Wang:
\begin{theorem}\label{thm5.1}[Theorem 1.4 in \cite{sung2014sharp}]
 Let $(M,g)$ be a complete manifold of dimension $n+1\geq 3$ with ${\rm Ric}\geq -ng$ and $\lambda_{1,p}(M)=\big(\frac{n}{p}\big)^p$ for
 some $p\leq\frac{n^2}{2(n-1)}.$ Then
  \par (1) $M$ is connected at infinity; or
  \par (2) $n\geq3$ and $(M,g)=(\mathbb{R}\times N,dt^2+e^{2t}g_N)$ where $N$ is compact; or
  \par (3) $n=2$ and $(M,g)=(\mathbb{R}\times N,dt^2+\cosh^2 (t)g_N)$ where $N$ is a compact.
\end{theorem}
In the following, we will show that the cases (2) and (3) would not happen under  the condition
of Proposition \ref{prop1.5}. We only need to calculate the $p-$eigenvalue of geodesic balls in the following two examples.
\begin{example}\label{ex5.2}
  Consider the manifold $(\mathbb{R}\times N,dt^2+e^{2t}g_N)$ where $N$ is compact of dimension $n\geq 3.$ Then for any point $o$ and any
  $p>1,$
  \begin{equation}\label{5.1}
    \lambda_{1,p}(B_o(R))\leq\Big(\frac{n}{p}\Big)^p+\Big(\frac{n}{p}\Big)^{p-2}\frac{p}{8}\frac{\pi^2}{R^2}+O(R^{-3})\ \ \ R\rightarrow\infty.
  \end{equation}
\end{example}
\begin{proof}
Let $B_N(R)=(-R,R)\times N$ for large $R>0,$ then
$$B_o(R-{\rm d}_g(o,N))\subseteq B_N(R)\subseteq B_o(R+{\rm diam}(N)+{\rm d}_g(o,N)).$$
Therefore, we only need to consider the $p-$eigenvalue of $B_N(R)$ as $R\rightarrow\infty.$
\par Set $f=e^{-\frac{nt}{p}}\cos(\frac{\pi t}{2R})\in W^{1,p}_0(B_N(R))$ for $t\in [-R,R].$ Then
\begin{equation}\label{5.2}
\begin{aligned}
   \lambda_{1,p}(B_N(R))& \leq\frac{\displaystyle\int_{B_N(R)}|\nabla f|^p\ d\upsilon_g}{\displaystyle\int_{B_N(R)}|f|^p d\upsilon_g}
=\frac{\displaystyle\int_{-R}^R|f'(t)\nabla t|^p \cdot{\rm Vol}(N)e^{nt}dt}{\displaystyle\int_{-R}^R|f(t)|^p \cdot{\rm Vol}(N)e^{nt}dt}
   \\ &=\frac{\displaystyle\int_{-R}^R \big|\frac{n}{p}\cos\frac{\pi t}{2R}+\frac{\pi}{2R}\sin\frac{\pi t}{2R}\big|^p dt}{\displaystyle\int_{-R}^R\cos^p\frac{\pi t}{2R}dt}
   \\ &=\Big(\frac{n^2}{p^2}+\frac{\pi^2}{4R^2}\Big)^{\frac{p}{2}}
   \end{aligned}
\end{equation}
Then
\begin{equation}\label{5.3}
\begin{aligned}
 \lambda_{1,p}(B_o(R))&\leq\lambda_{1,p}\bigl(B_N\big(R-{\rm diam}(N)-{\rm d}_g(o,N)\big)\bigl)
 \\&\leq \Big(\frac{n}{p}\Big)^p+\Big(\frac{n}{p}\Big)^{p-2}\frac{p}{8}\frac{\pi^2}{R^2}+O(R^{-3})\ \ \ R\rightarrow\infty
 \end{aligned}
 \end{equation}
\end{proof}

\begin{example}
  Consider the manifold $(\mathbb{R}\times N,dt^2+\cosh^2(t)g_N)$ where $N$ is compact of dimension $n=2.$ Then for any point $o$ and any
  $p>1,$
  \begin{equation}\label{5.4}
    \lambda_{1,p}(B_o(R))\leq\Big(\frac{2}{p}\Big)^p+\Big(\frac{2}{p}\Big)^{p-2}\frac{p}{8}\frac{\pi^2}{R^2}+O(R^{-3})\ \ \ R\rightarrow\infty.
  \end{equation}
\end{example}
The proof is almost the same as Example \ref{ex5.2} by choosing the test function as $f=\cosh^{-\frac{2}{p}}(t)\cos\frac{\pi t}{2R}.$
\begin{proof}[Proof of Proposition \ref{prop1.5}:]
 If $M$ is not connected at infinity, then (\ref{5.1}) must hold according to Theorem \ref{thm5.1} and the above two examples. So
$$\min\{p-1,1\}\leq \frac{p}{8}$$
and it contradicts the scope of $p.$
\end{proof}

\bibliographystyle{plain}%

\begin{thebibliography}{10}

\bibitem{allegretto1998picone}
Walter Allegretto and Huang~Yin Xi.
\newblock A Picone's identity for the $p-$Laplacian and applications.
\newblock {\em Nonlinear Analysis: Theory, Methods \& Applications},
  32(7):819--830, 1998.

\bibitem{anderson2008einstein}
Michael~T Anderson.
\newblock Einstein metrics with prescribed conformal infinity on 4-manifolds.
\newblock {\em Geometric and Functional Analysis}, 18(2):305--366, 2008.

\bibitem{barles1988remarks}
G~Barles.
\newblock Remarks on uniqueness results of the first eigenvalue of the $ p
  $-Laplacian.
\newblock In {\em Annales de la Facult{\'e} des sciences de Toulouse:
  Math{\'e}matiques},  9:65--75, 1988.

\bibitem{belloni2002direct}
Marino Belloni and Bernd Kawohl.
\newblock A direct uniqueness proof for equations involving the $p-$Laplace
  operator.
\newblock {\em Manuscripta mathematica}, 109(2), 2002.

\bibitem{besse2007einstein}
Arthur~L Besse.
\newblock  Einstein manifolds.
\newblock Springer, 2007.

\bibitem{cai1999boundaries}
Mingliang Cai and Gregory~J Galloway.
\newblock Boundaries of zero scalar curvature in the AdS/CFT correspondence.
\newblock {\em Advances in theoretical and mathematical physics},
  3(6):1769--1783, 1999.

\bibitem{carvalho2022fundamental}
Francisco G de~S Carvalho and Marcos P de~A Cavalcante.
\newblock On the fundamental tone of the $p-$Laplacian on Riemannian manifolds
  and applications.
\newblock {\em Journal of Mathematical Analysis and Applications},
  506(2):125703, 2022.

\bibitem{cheng1975eigenvalue}
Shiu-Yuen Cheng.
\newblock Eigenvalue comparison theorems and its geometric applications.
\newblock {\em Mathematische Zeitschrift}, 143:289--297, 1975.

\bibitem{du2017estimates}
Feng Du and Jing Mao.
\newblock Estimates for the first eigenvalue of the drifting Laplace and the
  $p-$Laplace operators on submanifolds with bounded mean curvature in the
  hyperbolic space.
\newblock {\em Journal of Mathematical Analysis and Applications},
  456(2):787--795, 2017.

\bibitem{graham1999volume}
C~Robin Graham.
\newblock Volume and area renormalizations for conformally compact Einstein
  metrics.
\newblock {\em Proc. of 19th Winter School in Geometry and Physics, Srni, Czech
  Rep., Jan.}, 9909042, 1999.

\bibitem{graham1991einstein}
C~Robin Graham and John~M Lee.
\newblock Einstein metrics with prescribed conformal infinity on the ball.
\newblock {\em Advances in mathematics}, 87(2):186--225, 1991.

\bibitem{henrot2006extremum}
Antoine Henrot.
\newblock  Extremum problems for eigenvalues of elliptic operators.
\newblock Springer Science \& Business Media, 2006.

\bibitem{hijazi2020cheeger}
Oussama Hijazi, Sebasti{\'a}n Montiel, and Simon Raulot.
\newblock The Cheeger constant of an asymptotically locally hyperbolic manifold
  and the Yamabe type of its conformal infinity.
\newblock {\em Communications in Mathematical Physics}, 374(2):873--890, 2020.

\bibitem{jin2021finite}
Xiaoshang Jin.
\newblock Finite boundary regularity for conformally compact Einstein manifolds
  of dimension 4.
\newblock {\em The Journal of Geometric Analysis}, 31:4004--4023, 2021.

\bibitem{jin2024asymptotic}
Xiaoshang Jin.
\newblock Asymptotic behavior of the first Dirichlet eigenvalue of AHE
  manifolds.
\newblock {\em Bulletin of the London Mathematical Society}, 56(6):2024--2035,
  2024.

\bibitem{jin2024rigidity}
Xiaoshang Jin.
\newblock Rigidity for inscribed radius estimate of asymptotically hyperbolic
  Einstein manifold.
\newblock {\em Proceedings of the American Mathematical Society},
  152(12):5327--5337, 2024.

\bibitem{kasue1983ricci}
Atsushi Kasue.
\newblock Ricci curvature, geodesics and some geometric properties of
  Riemannian manifolds with boundary.
\newblock {\em Journal of the Mathematical Society of Japan}, 35(1):117--131,
  1983.

\bibitem{kristaly2022new}
Alexandru Krist{\'a}ly.
\newblock New features of the first eigenvalue on negatively curved spaces.
\newblock {\em Advances in Calculus of Variations}, 15(3):475--495, 2022.

\bibitem{lee1995spectrum}
John~M Lee.
\newblock The spectrum of an asymptotically hyperbolic Einstein manifold.
\newblock {\em Communications in Analysis and Geometry}, 3(2):253--271, 1995.

\bibitem{li2015rigidity}
Haizhong Li and Yong Wei.
\newblock Rigidity theorems for diameter estimates of compact manifold with
  boundary.
\newblock {\em International Mathematics Research Notices},
  2015(11):3651--3668, 2015.

\bibitem{li2002complete}
Peter Li and Jiaping Wang.
\newblock Complete manifolds with positive spectrum, ii.
\newblock {\em Journal of Differential Geometry}, 62(1):143--162, 2002.

\bibitem{maz2011sobolev}
Vladimir~G  Maz'ya.
\newblock {\em Sobolev Spaces: With Applications to Elliptic Partial
  Differential Equations}.
\newblock Springer, 2011.

\bibitem{mazzeo1988hodge}
Rafe Mazzeo.
\newblock The Hodge cohomology of a conformally compact metric.
\newblock {\em Journal of differential geometry}, 28(2):309--339, 1988.

\bibitem{perez2024first}
Samuel P{\'e}rez-Ayala and Aaron~J Tyrrell.
\newblock First eigenvalue estimates for asymptotically hyperbolic manifolds
  and their submanifolds.
\newblock {\em arXiv preprint arXiv:2404.07365}, 2024.

\bibitem{poliquin2014bounds}
Guillaume Poliquin.
\newblock Bounds on the principal frequency of the $p-$Laplacian.
\newblock {\em Contemporary Mathematics}, 630:349--366, 2014.

\bibitem{savo2009lowest}
Alessandro Savo.
\newblock On the lowest eigenvalue of the Hodge Laplacian on compact,
  negatively curved domains.
\newblock {\em Annals of Global Analysis and Geometry}, 35:39--62, 2009.

\bibitem{sung2014sharp}
Chiung-Jue~Anna Sung and Jiaping Wang.
\newblock Sharp gradient estimate and spectral rigidity for $ p $-Laplacian.
\newblock {\em Mathematical Research Letters}, 21(4):885--904, 2014.

\bibitem{takeuchi1998first}
Hiroshi Takeuchi.
\newblock On the first eigenvalue of the $ p $-Laplacian in a Riemannian
  manifold.
\newblock {\em Tokyo Journal of Mathematics}, 21(1):135--140, 1998.

\bibitem{wang2001conformally}
Xiaodong Wang.
\newblock On conformally compact Einstein manifolds.
\newblock {\em Mathematical Research Letters}, 8(5):671--688, 2001.

\bibitem{wang2002new}
Xiaodong Wang.
\newblock A new proof of Lee’s theorem on the spectrum of conformally compact
  Einstein manifolds.
\newblock {\em Communications in Analysis and Geometry}, 10(3):647--651, 2002.

\bibitem{witten1999connectedness}
Edward Witten and Shing-Tung Yau.
\newblock Connectedness of the boundary in the AdS/CFT correspondence.
\newblock {\em Advances in Theoretical and Mathematical Physics}, 3(6):1-19, 1999.

\end{thebibliography}

\noindent{Xiaoshang Jin}\\
  School of mathematics and statistics, Huazhong University of science and technology, Wuhan, P.R. China. 430074
 \\Email address: jinxs@hust.edu.cn

\end{document}